\documentclass[final]{siamonline250211}
\usepackage{amssymb,amsfonts,amsmath}
\usepackage[mathscr]{eucal}
\usepackage{amsbsy}
\usepackage{stmaryrd}
\usepackage{bm}
\usepackage{braket}
\usepackage{xparse}
\usepackage{xspace}
\usepackage{enumitem}
\usepackage{booktabs}
\usepackage{xcolor}
\usepackage{hyperref}
\usepackage{cleveref}
\crefformat{footnote}{#2\footnotemark[#1]#3}
\SetSymbolFont{stmry}{bold}{U}{stmry}{m}{n}
\newsiamremark{remark}{Remark}
\AddToHook{env/remark/begin}{\crefalias{theorem}{remark}}
\crefname{remark}{Remark}{Remarks}

\renewcommand{\epsilon}{\varepsilon}
\NewDocumentCommand \qtext {m} {\quad\text{#1}\quad}
\NewDocumentCommand \Real {} {\mathbb{R}}
\NewDocumentCommand \Natural {} {\mathbb{N}}
\NewDocumentCommand{\Tr}{s}{\IfBooleanTF{#1}{\vphantom{\intercal}}{\intercal}}
\NewDocumentCommand{\Mx}{O{} m !g t' t"}{
  \bm{#1{\mathbf{\MakeUppercase{#2}}}}%
  \IfValueT{#3}{_{#3}}%
  \IfBooleanTF{#4}{^{\Tr}}{%
    \IfBooleanT{#5}{^{\Tr*}}}%
}
\NewDocumentCommand \V { O{} m } {{\bm{#1\mathbf{\MakeLowercase{#2}}}}} 
\NewDocumentCommand \B { t' t~ t*} {\Mx[\IfBooleanT{#2}{\tilde}]{B}\IfBooleanT{#1}{^{\intercal}}\IfBooleanT{#3}{_{\ast}}}
\NewDocumentCommand \A { } {\Mx{A}}
\NewDocumentCommand \Om { t! } {\Mx[\IfBooleanT{#1}{\bar}]{\Omega}}
\NewDocumentCommand \Vp { } {\V{p}}
\NewDocumentCommand \Multi { } {\text{\sc multinomial}}
\NewDocumentCommand \RandSample { } {\text{\sc randsample}}
\NewDocumentCommand \HybridSample { } {\text{\sc hybridsample}}
\NewDocumentCommand \Prob { } {\text{\rm Pr}}
\NewDocumentCommand \Exp { } {\mathbb{E}}
\DeclareDocumentCommand \det { } {\text{\rm\sffamily det}}
\NewDocumentCommand \pdet { } {p_{\text{\rm\sffamily det}}} 
\NewDocumentCommand \DetSet { } {\mathcal{D}} %
\NewDocumentCommand{\resid}{}{\mathcal{R}} 
\NewDocumentCommand{\Sdet}{}{\Mx{D}}
\NewDocumentCommand{\Srnd}{}{\Mx[\tilde]{S}}
\newcommand{\trans}{\intercal}
\DeclareMathOperator*{\argmin}{arg\,min}
\DeclareMathOperator{\rank}{rank}
\newcommand{\R}{\mathbb{R}}
\newcommand{\E}{\mathbb{E}}
\newcommand{\Var}{\mathrm{Var}}
\newcommand{\Bperp}{\ensuremath{\Mx{B}^{\perp}}}
\newcommand{\UA}{\ensuremath{\Mx{U}{\Mx{A}}}}
\newcommand{\Xopt}{\ensuremath{\Mx{X}_{\text{\rm\sffamily opt}}}}
\newcommand{\Xtildeopt}{\ensuremath{\widetilde{\Mx{X}}_{\text{\rm\sffamily opt}}}}
\newcommand{\Ytildeopt}{\ensuremath{\widetilde{\Mx{Y}}_{\text{\rm\sffamily opt}}}}
\newcommand{\Yopt}{\ensuremath{\Mx{Y}_{\text{\rm\sffamily opt}}}}

\title{Revisiting Approximate Leverage Score Sketching for Matrix Least Squares\thanks{v1 of this manuscript was intended to provide a detailed and standalone derivation of a result that appeared in~\cite[Appendix A]{LaKo22}. Subsequent versions provide improved bounds and expand the analysis to hybrid sampling.}}

\author{%
  Brett W. Larsen%
  \thanks{Redwood City, CA (\email{brett.larsen14@gmail.com})}
  \and Tamara~G.~Kolda%
  \thanks{MathSci.ai, Dublin, CA (\email{tgkolda@mathsci.ai})}
}

\headers{Sketching Matrix Least Squares via Leverage Scores Estimates}
{Brett W. Larsen and Tamara G. Kolda}

\begin{document}

\maketitle

\begin{abstract}
    We revisit the problem of sketching using approximate leverage scores
    for matrix least squares problems of the form
    $\| \A\Mx{X}-\Mx{B} \|_F^2$ where the design matrix
    $\A \in \Real^{N \times r}$ is tall and skinny with  $N \gg r$.
    We derive the theoretical results from first principles and clarify the relation to previously stated bounds, improving
    some constants along the way. 
    One can characterize the utility of a sketching scheme according
    to the number of samples it needs for an $\varepsilon$-accurate solution 
    with high probability.
    Assuming $\varepsilon$ is suitably small, we will show that 
    approximate leverage score sampling requires $4r/(\beta\delta\epsilon)$
    samples, where $\delta$ is the failure probability and 
    $\beta \in (0,1]$ is a measure of the quality of the approximate 
    leverage scores such that $\beta=1$ corresponds to using exact leverage 
    scores.
    In cases where a few approximate leverage scores are very large 
    (summing to $\pdet$),
    we also show that using a hybrid deterministic and random sampling
    scheme reduces the required number
    of samples by a factor  of $1/(1-\pdet)$.
\end{abstract}


\section{Introduction}%
\label{sec:intro}%

Approximating the solution of an overdetermined system of linear equations is a fundamental problem in data science and statistics.
This is often accomplished via the method of least squares which finds the matrix $\Xopt := \argmin_{\Mx{X}} \| \A \Mx{X} -\Mx{B} \|_F^2$.
Here we consider the problem of sketching this matrix least squares problem
by row sampling according to some probability distribution $\V{p}$.
The sketched matrices $\Mx[\widetilde]{A}$ and $\Mx[\widetilde]{B}$
are weighted subsets of the rows of the original matrices.
Our aim is for the solution to the sketched problem $\Xtildeopt := \argmin \| \Mx[\widetilde]{A}\Mx{X}-\Mx[\widetilde]{B} \|_F^2$ to be $\epsilon$-accurate,
meaning that its residual satisfies the following property:
\begin{equation*}
    \| \A\Xtildeopt-\Mx{B} \|_F^2
    \leq (1+\epsilon)  \| \A \Xopt -\Mx{B} \|_F^2.
\end{equation*}

If we sample proportional to approximate leverage scores, then we can obtain an $\epsilon$-accurate solution with high probability using $O(r/(\beta \epsilon))$ samples, where $\beta \in (0,1]$ is a measure of the quality of the leverage score estimates. Specifically,
for $\A \in \Real^{N \times r}$,
we prove that this occurs with probability at least $1- \delta$  provided the number of sampled rows is
\begin{displaymath}
    s=({r}/{\beta}) \max \set{ C \log (2r/\delta), {4}/(\delta \epsilon) }
\end{displaymath}
where $C \approx 20.9080$.
If $\epsilon$ is sufficiently small so that $\epsilon^{-1} \geq C \delta \log(r/\delta)$, 
we have $s=4r/(\beta\delta\varepsilon)$.

We also consider a hybrid sampling scheme where we deterministically include the rows with the largest approximate leverage scores and then sample from the remaining rows as before.
If $d$ is the number of deterministically included rows and $\pdet$ is the sum of the probabilities corresponding to those rows,
then the number of samples is
\begin{displaymath}
    s = d + \frac{1}{(1-\pdet)} (r/\beta)\max \set{ 2C \log (2r/\delta), {4}/(\delta \epsilon) }.
\end{displaymath}
This effectively reduces the number of samples by a factor of $1/(1-\pdet)$, which can make a major difference in practical implementations.

This note provides complete proofs of these results with several motivations in mind.
The first motivation is that this work provides the foundation for leverage-based sampling for low-rank tensor decomposition as described in \cite{LaKo22}.
To the best of our knowledge, the precise result stated in Theorem 6 of that paper was new and thus a condensed outline of the proof was provided in the appendix.
While the original version of this note \cite{arXiv-LaKo22} provided an extended proof,
here we provide a revised version with improved constants and a more complete derivation.
The second motivation is that although many of the steps are from previous work or primarily extend results to the matrix case, we did not find a concise statement of the full logic of this style of sketching least squares results.
The third is proving the result for the hybrid sampling scheme
in the case of approximate leverage scores, extending the results of Hayashi et al.~\cite{hayashi2025randomized} in the case of exact leverage scores.
We show that the hybrid sampling scheme improves the sample complexity by a factor of $1/(1-\pdet)$.

\subsection{Problem Setup and Notation}
\label{sec:problem-setup}

Consider the overdetermined matrix least squares problem defined by the \emph{design matrix} $\A \in \R^{N \times r}$,
with $N > r$ and  $\rank(\A) = r$, and the matrix $\Mx{B} \in \R^{N \times n}$.
Define the solution and optimal squared residual to be
\begin{equation}\label{eq:resid}
    \Xopt \triangleq \argmin_{\Mx{X} \in \R^{r \times n}} \|\A \Mx{X} - \Mx{B}\|_F^2
    \qtext{and}
    \resid^2 \triangleq \|\A \Xopt - \Mx{B}\|_F^2.
\end{equation}
The SVD of the design matrix is $\A = \UA \Mx{\Sigma}{\A} {\Mx{V}{\!\!\A}}^{\!\!\Tr}$,
so $\UA$ is an orthonormal basis for the $r$-dimensional column space of $\A$.
Let $\UA^{\perp}$ be an orthonormal basis for the $(N-r)$-dimensional
subspace orthogonal to the column space of $\A$.
We define $\Bperp$ to be the projection of the the columns of $\Mx{B}$ onto this orthogonal subspace.
This matrix is important because the residual of the least squares problem is its Frobenious norm;
$\Mx{X}$ can be chosen so that each column in $\A \Mx{X}$ exactly matches the part of the corresponding column in $\Mx{B}$
in the column space of $\A$ but cannot, by definition, match anything in the range spanned by $\UA^{\perp}$:
\begin{equation}\label{eq:bperp}
    \Bperp \triangleq \UA^{\perp} \UA^{\perp \trans} \Mx{B}
    \quad \Rightarrow \quad
    \Mx{B} = \A \Xopt + \Bperp
    \qtext{and}
    \resid^2
    = \|\Bperp\|_F^2.
\end{equation}

\subsection{Outline}
\Cref{sec:proof-structure} reviews the standard theory approach for sketching least squares problems \cite{DrMaMaWo12}, reducing the problem to showing that two structural conditions holds.
\Cref{sec:proof-that-leverage} describes approximate leverage score sampling in detail and proves that the two structural conditions
hold in that case.
\Cref{sec:hybrid-sampling} describes the hybrid deterministic and random sampling approach and proves that the two structural
conditions also hold in that case.

\section{Analysis for Generic Sketching Matrix}
\label{sec:proof-structure}%
An $s \times N$ sketching matrix $\Mx{S}$ is applied to both $\A$ and $\Mx{B}$ to yield the sketched problem:
\begin{equation}\label{eq:lsq-sketched}
    \min_{\Mx{X} \in \R^{r \times n}} \|\Mx{S}\A \Mx{X} - \Mx{S}\Mx{B}\|_F^2.
\end{equation}
Following the technique of Drineas et al.~\cite{drineas2011faster}, we split the proof into two parts.
In the remainder of this section, we prove bounds on both the residual and the solution of the sketched system
for a \emph{generic} sketching matrix $\Mx{S}$ that satisfies certain \emph{structural conditions}.
The proofs follow deterministically and do not consider the random aspect of the sketching matrix generation.

In subsequent sections, we consider two different methods for generating the sketching matrix $\Mx{S}$ and show that these satisfy the structural conditions with high probability and give bounds on the number of samples required.

\subsection{Structural Conditions}
\label{sec:struct-conds}%

We begin by assuming that our design matrix satisfies two \emph{structural conditions}.
As defined in \cref{sec:problem-setup},
the matrix $\UA$ is an orthonormal basis for the column space of $\A$, so its singular values are all equal to 1.
The first structural condition says that the sketching matrix $\Mx{S}$ approximately maintains the orthogonality of $\UA$ in that the smallest singular value does not become too small:
\begin{equation}
    \tag{SC1} \label{eq:sc1}
    \sigma^2_{\text{min}}(\Mx{S} \UA) \geq 1/\sqrt{2}.
\end{equation}
The second structural condition has to do with the effect of sketching on residual. By definition of $\Bperp$ in \cref{eq:bperp}, its columns are orthogonal to the column space of $\A$, so $\UA'\Bperp = \Mx{0}$.
The second structural condition requires that after sketching, this orthogonality is approximately maintained in that the norm of the product $\UA' \Mx{S}' \Mx{S} \Bperp$ is small:
\begin{equation}
    \tag{SC2} \label{eq:sc2}
    \|\UA' \Mx{S}' \Mx{S} \Bperp \|_F^2 \leq \epsilon \resid^2/2.
\end{equation}
\subsection{Theory Under Structural Conditions}
\label{sec:prop-sketch-matr}%

The first result is analogous to Drineas et al.~\cite[Lemma~1]{drineas2011faster}
except that we prove it for the \emph{matrix} least squares case.
This result shows that if the sketching matrix $\Mx{S}$ satisfies the structural conditions,
then the solution to the sketched problem has both a small residual and is $\varepsilon$-accurate.
\begin{theorem}\label{thm:sketch-bound}
    For the overdetermined least squares problem \cref{eq:lsq-sketched}, assume
    the sketch matrix $\Mx{S}$ satisfies \cref{eq:sc1,eq:sc2} for some
    $\epsilon \in (0, 1)$.  Then the solution to the sketched problem,
    denoted $\Xtildeopt$, satisfies the following two bounds:
    \begin{align*}
        \|\A \Xtildeopt - \Mx{B} \|_F^2 & \leq (1 + \epsilon) \|\A \Xopt - \Mx{B} \|_F^2, \qtext{and}                      \\
        \|\Xopt - \Xtildeopt \|_F^2     & \leq \frac{\epsilon \|\A \Xopt - \Mx{B} \|_F^2}{\sigma^2_{\text{\rm min}}(\A)} .
    \end{align*}
\end{theorem}

\begin{proof}
    We begin by rewriting the sketched regression problem:
    \begin{align*}
        \min_{\Mx{X} \in \R^{r \times n}} \|\Mx{S}\A \Mx{X} - \Mx{S}\Mx{B}\|_F^2
         & = \min_{\Mx{X} \in \R^{r \times n}} \|\Mx{S}\A \Mx{X} - \Mx{S}\A \Xopt + \Mx{S}\A \Xopt - \Mx{S}\Mx{B}\|_F^2 \\
         & = \min_{\Mx{X} \in \R^{r \times n}} \|\Mx{S}\A (\Mx{X} - \Xopt) - \Mx{S}(\A \Xopt + \Bperp)\|_F^2            \\
         & = \min_{\Mx{Y} \in \R^{r \times n}} \|\Mx{S}\UA (\Mx{Y} - \Yopt) - \Mx{S}\Bperp\|_F^2.
    \end{align*}
    In the last line, we reparameterize the matrices $\Mx{X}$ and
    $\Xopt$ in terms of the orthonormal basis $\UA$ such that
    $\UA \Mx{Y} = \A \Mx{X}$ and the analogous relationships hold for
    $\Xopt/\Yopt$ and $\Xtildeopt/\Ytildeopt$.
    The solution, $\Ytildeopt$, satisfies the normal equation, i.e.,
    \begin{equation*}
        (\Mx{S}\UA)^{\trans} \Mx{S}\UA (\Ytildeopt - \Yopt) = (\Mx{S}\UA)^{\trans} \Mx{S} \Bperp.
    \end{equation*}
    By \cref{eq:sc1}, we have that $\sigma_i((\Mx{S}\UA)^{\trans} \Mx{S}\UA) = \sigma_i^2(\Mx{S}\UA) \geq 1/\sqrt{2}$.
    Thus taking the norm squared of both sides,
    applying the structural conditions, and then the relation from the normal equation gives:
    \begin{equation*}
        \|\Ytildeopt - \Yopt\|_F^2 /2 \leq
        \|(\Mx{S}\UA)^{\trans} \Mx{S}\UA (\Ytildeopt - \Yopt)\|_F^2
        = \|(\Mx{S}\UA)^{\trans} \Mx{S} \Bperp\|_F^2.
    \end{equation*}
    Applying \cref{eq:sc2} to the right hand side of this inequality yields
    \begin{align}
        \notag
        \|\Ytildeopt - \Yopt\|_F^2 /2 \leq \|\UA' \Mx{S}' \Mx{S} \Bperp\|_F^2 & \leq \epsilon \resid^2/2, \\
        \label{eq:orthBound}
        \Longrightarrow \|\Ytildeopt - \Yopt\|_F^2                            & \leq \epsilon \resid^2.
    \end{align}
    We can then obtain the desired result on the residual:
    \begin{align*}
        \| \Mx{B} - \A \Xtildeopt \|_F^2 & = \|\Mx{B} - \A \Xopt + \A \Xopt - \A \Xtildeopt \|_F^2,                       \\
                                         & = \|\Mx{B} - \A \Xopt \|_F^2 + \|\A (\Xopt -  \Xtildeopt) \|_F^2,              \\
                                         & = \|\Bperp\|_F^2 + \|\UA (\Yopt -  \Ytildeopt) \|_F^2
        = \|\Bperp\|_F^2 + \|\Yopt -  \Ytildeopt \|_F^2,                                                                  \\
                                         & \leq \resid^2 + \epsilon \resid^2 = (1 + \epsilon) \|\Mx{B} - \A \Xopt \|_F^2,
    \end{align*}
    where we have used in line 2 that the columns of $\Mx{B} - \A \Xopt = \Bperp$ are orthogonal to
    $\A$ times any vector and in the third line that $\UA$ is a matrix with orthonormal columns.

    Lastly, to obtain the bound on the solution recall that $\A(\Xopt - \Xtildeopt) = \UA(\Yopt - \Ytildeopt)$.
    Taking the norm of both sides we have:
    \begin{equation*}
        \sigma^2_{\text{min}}(\A) \|(\Xopt - \Xtildeopt)\|_F^2 \leq \|\A(\Xopt - \Xtildeopt)\|_F^2 = \|\UA(\Yopt - \Ytildeopt)\|_F^2.
    \end{equation*}
    Recall that we assume $\rank(\A)=r$ so that $\sigma_{\min}(\A)>0$.
    We then apply \cref{eq:orthBound} and rearrange to obtain the desired result:
    \begin{displaymath}
        \|(\Xopt - \Xtildeopt)\|_F^2
        \leq \frac{\|(\Yopt - \Ytildeopt)\|_F^2}{\sigma^2_{\text{min}}(\A)}
        \leq \frac{\epsilon \resid^2}{\sigma^2_{\text{min}}(\A)}.
    \end{displaymath}
\end{proof}

We can obtain a tighter bound on the solution matrix if we assume a constant fraction of
the columns of $\Mx{B}$ is in the column space of $\A$.  This is typically a
reasonable assumption for real-world least squares problems as the fit
is only practically interesting if this is true.

\begin{theorem}[Drineas et al.~\cite{drineas2011faster}]
    Assume that the conditions of \cref{thm:sketch-bound} hold.
    Furthermore, assume that $\|\UA \UA' \Mx{B} \|_F \geq \gamma \|\Mx{B}\|_F$ for some fixed $\gamma \in (0, 1]$.
    Then the solution to the sketched problem, denoted $\Xtildeopt$, satisfies the following bound:
    \begin{equation*}
        \|\Xopt - \Xtildeopt \|_F^2 \leq \epsilon \kappa(\A)^2 (\gamma^{-2} - 1) \|\Xopt\|_F^2,
    \end{equation*}
    where $\kappa(\A)$ denotes the condition number of the matrix $\A$.
\end{theorem}

\begin{proof}
    Start by bounding the residual squared using our assumption on $\Mx{B}$ as follows:
    \begin{align*}
        \|\A \Xopt - \Mx{B} \|_F^2 = \|\Bperp\|_F^2 & = \|\Mx{B}\|_F^2 - \|\UA \UA' \Mx{B} \|_F^2,                          \\
                                                    & \leq \gamma^{-2} \|\UA \UA' \Mx{B} \|_F^2 - \|\UA \UA' \Mx{B} \|_F^2, \\
                                                    & = (\gamma^{-2} - 1) \|\UA \UA' \Mx{B} \|_F^2,                         \\
                                                    & = (\gamma^{-2} - 1) \|\A \Xopt \|_F^2,                                \\
                                                    & \leq \sigma^2_{\text{max}}(\A) (\gamma^{-2} - 1) \|\Xopt \|_F^2.
    \end{align*}
    By the previous theorem, we have that $\|\Xopt - \Xtildeopt\|_F^2 \leq \frac{1}{\sigma^2_{\text{min}}(\A)} \epsilon \|\A \Xopt - \Mx{B} \|_F^2$. Plugging in the above inequality yields the desired result:
    \begin{align*}
        \|\Xopt - \Xtildeopt\|_F^2
         & \leq \frac{1}{\sigma^2_{\text{min}}(\A)} \epsilon \|\A \Xopt - \Mx{B} \|_F^2,                                \\
         & \leq \epsilon \frac{\sigma^2_{\text{max}}(\A)}{\sigma^2_{\text{min}}(\A)}  (\gamma^{-2} - 1) \|\Xopt \|_F^2, \\
         & = \epsilon \kappa(\A)^2 (\gamma^{-2} - 1) \|\Xopt\|_F^2 .
    \end{align*}
\end{proof}

\section{Approximate Leverage Score Sampling}
\label{sec:proof-that-leverage}%
In this section, we show that the methodology for choosing the columns
via the leverage-score-based sampling scheme satisfies the structural conditions.

\subsection{Weighted Sampling}
\label{sec:weighted-sampling}
Assuming we choose rows of a matrix according to some probability distribution, we consider how to weight the rows
so that the subsampled norm is unbiased.

\begin{definition}
    We say $\Vp \in [0,1]^N$ is a \emph{probability distribution} if
    $\sum_{i=1}^N p_i = 1$.
\end{definition}

\begin{definition}
    For a random variable $\xi \in [N]$, we say $\xi \sim \Multi(\Vp)$ if $\Vp \in [0,1]^N$ is a probability distribution and $\Prob(\xi = i) = p_i$.
\end{definition}

We can define a matrix that randomly samples rows from a matrix (or elements from a vector) with weights as follows.
The following definition can be found, e.g., in Woodruff \cite[Defn.~16]{Wo14} or Drineas and Mahoney \cite[Alg.~1]{DrMa17}.

\begin{definition}\label{def:randsample}
    We say $\Mx{S} \in \Real^{s \times N} \sim \RandSample(s,\Vp)$
    if $s \in \Natural$,
    $\Vp \in [0,1]^N$ is a probability distribution,
    and the entries on $\Om$ are defined as follows.
    Let $\xi^{(j)} \sim \Multi(\Vp)$ for $j=1,\dots,s$; then
    \begin{equation}\label{eq:S}
        \Mx{S}(j,i) =
        \begin{cases}
            \frac{1}{\sqrt{s p_i}} & \text{if } \xi^{(j)} = i, \\
            0                      & \text{otherwise},
        \end{cases}
        \qtext{for all}
        (j,i) \in [s] \times [N].
    \end{equation}
\end{definition}

\noindent It is straightforward to show that such a sampling matrix is unbiased, so we leave the proof of the next lemma as an exercise for the reader.

\begin{lemma}\label{lem:expectation}
    Let $\V{x} \in \Real^N$.
    Let $\Vp \in [0,1]^N$ be probability distribution such that $p_i > 0$ if $x_i \neq 0$
    and let $\Om \sim \RandSample(s,\Vp)$.
    Then $\Exp{ \| \Mx{S} \V{x} \|_2^2} = \|\V{x}\|_2^2$.
\end{lemma}

\subsection{Leverage Scores and Approximate Leverage Score Sampling}
\label{sec:leverage-scores}
The distribution selected for $\Vp$ determines the quality of the estimate in a way that depends on the leverage scores of $\A$.
We use the presentation of leverage scores from Drineas et al.~\cite{DrMaMaWo12}.

\begin{definition}[Leverage Scores \cite{DrMaMaWo12}]\label{def:leverages_scores}
    Let $\A \in \Real^{N \times r}$ with $N > r$, and
    let $\UA \in \Real^{N \times r}$ be any orthogonal basis for the column space of $\A$ (such as the left singular vectors from the compact SVD).
    The \emph{leverage scores} of the rows of $\A$ are given by
    \begin{displaymath}
        \ell_i(\UA) = \|\UA(i,:)\|_2^2 \qtext{for all} i \in \set{1,\dots,N}.
    \end{displaymath}
    The \emph{coherence} is the maximum leverage score, denoted
    $\mu(\A) = \max_{i \in [N]} \ell_i(\A)$.
\end{definition}

\noindent The leverage scores indicate the relative importance of rows in the matrix $\A$.
It is known that $\ell_i(\A) \leq 1$ for all $i \in [N]$, $\sum_{i\in[N]} \ell_i(\A) = r$,  and $\mu(\A) \in [r/N,1]$ \cite{Wo14}.
The matrix $\A$ is called incoherent if $\mu(\A) \approx r/N$.
Since the leverage scores sum to $r$, sampling row $i$ according to its leverage score means it is sampled with probability $\ell_i(\A)/r$.
This distribution is the optimal distribution, but calculating true
leverage scores is as expensive as solving the original least squares problem. Thus, we generally work with \emph{approximate leverage scores}.

For any row sampling distribution $\V{p}$ we can measure the discrepancy between it and the sampling distribution defined by the leverage scores via the \textit{misestimation} factor $\beta \in (0, 1]$ such that
\begin{displaymath}
    p_i \geq \beta \; {\ell_i(\A)}/{r} \qtext{for all} i \in [N].
\end{displaymath}
This means that $p_i$ is at least a $\beta$-fraction of the optimal leverage score based sampling probability for all rows.
When $\beta = 1$, we have exact leverage score sampling.
When $\beta$ is small, the sampling distribution may be poor in that it undersamples important rows.

\subsection{First Structural Condition for Random Sampling}
\label{sec:sc1}

The first structural condition \cref{eq:sc1} is shown using a matrix Chernoff bound. We state precisely the Chernoff bound for completeness. (We only require the lower bound.)

\begin{theorem}[Matrix Chernoff~{\cite[Corollary of Theorem~5.1.1]{Tropp15}}]\label{thm:matrixChernof}
    Consider a finite sequence $\{\Mx{X}_k\}$ of independent, random, positive-semidefinite matrices
    of dimension $d$ satisfying $\lambda_{\max}(\Mx{X}_k) \leq L$ and $\sum_k \Exp\Mx{X}_k = \Mx{I}$. Then, for any $\varepsilon \in [0,1)$, we have
    \begin{displaymath}
        \Prob\left\{ \lambda_{\min}\left({\textstyle\sum_k \Mx{X}_k}\right) \leq (1-\varepsilon) \right\}
        \leq d \left(\frac{e^{-\varepsilon}}{(1-\varepsilon)^{1-\varepsilon}}\right)^{1/L}.
    \end{displaymath}
\end{theorem}

Using the Chernoff bound, we show  in \cref{lem:sc1}
that the first structural condition holds
with high probability for approximate leverage score sampling.
\Cref{rem:woodruff-compare} provides context for why the constants differ compared similar results in the literature, including those used in the original version of this artricle.
The proof is a straightforward application of the matrix Chernoff bound and is closely related to previous results; see \cref{rem:sc1-previous}.

\begin{lemma}[SC1 for Random Sampling]\label{lem:sc1}
    Consider $\A \in \R^{N \times r}$, its SVD $\UA \Mx{\Sigma}_{\A} \Mx{V}{\A}'$,
    and row leverage scores $\ell_i(\A)$.
    Let $\Vp$ be a probability distribution such that, for some positive $\beta \leq 1$,
    we have $p_i \geq \beta \ell_i(\A)/r$ for all $i \in [N]$.
    Let $\Mx{S} \in \R^{s \times N} \sim \RandSample(s,\Vp)$.
    Then,
    \begin{displaymath}
        s > C \frac{r}{\beta} \ln(r/\delta)
        \quad \Rightarrow \quad
        \Prob\left\{ \sigma_{\min}^2(\Mx{S} \UA) \geq 1/\sqrt{2} \right\} \geq 1-\delta
    \end{displaymath}
    where $C=\sqrt{2}/(\sqrt{2}-1-\log(2)/2) \approx 20.9080$.
\end{lemma}

\begin{proof}
    We will apply the matrix Chernoff bound (\cref{thm:matrixChernof}) with
    a particular choice of $\varepsilon$ to obtain the desired result.
    First, we set up the conditions to apply the result.

    Let $\{\xi^{(1)}, \ldots, \xi^{(s)}\}$ be the indices of the $s$ sampled rows, drawn according to $\Vp$. Define the positive semidefinite random matrices
    \begin{equation*}
        \Mx{X}_k = \frac{1}{s \, p_{\xi^{(k)}}} \UA(\xi^{(k)}, :)^{\trans} \UA(\xi^{(k)}, :) \in \R^{r \times r}.
    \end{equation*}
    Note that $\sum_{k=1}^s \Mx{X}_k = (\Mx{S} \UA)^{\trans} \Mx{S} \UA$ by construction of the sampling matrix $\Mx{S}$ in \cref{eq:S}, and that the expectation of this sum is $\Mx{I}_r$ since
    \begin{align*}
        \sum_{k=1}^s \Exp[\Mx{X}_k] & = \sum_{k=1}^s \sum_{j=1}^N p_j \frac{1}{s \, p_j} \UA(j, :)^{\trans} \UA(j, :)
        = \UA' \UA = \Mx{I}_r.
    \end{align*}
    Additionally,
    using the assumption $p_i \geq \beta \ell_i(\A) / r$,
    the eigenvalue bound is
    \begin{equation*}
        \lambda_{\max}(\Mx{X}_k) = \frac{\|\UA(\xi^{(k)}, :)\|_2^2}{s \, p_{\xi^{(k)}}}
        = \frac{\ell_{\xi^{(k)}}(\A)}{s \, p_{\xi^{(k)}}} \leq \frac{r}{\beta s}.
    \end{equation*}
    Hence, we satisfy the conditions to apply the matrix Chernoff bound with
    $L = r/(\beta s)$ and $\mu_{\min} = 1$,
    yielding the lower tail bound for $\varepsilon \in [0,1)$:
    \begin{equation*}
        \Prob\left\{ \lambda_{\min}((\Mx{S}\UA)^{\trans} \Mx{S}\UA) \leq 1 - \varepsilon \right\}
        \leq r \left( g(\varepsilon) \right)^{\beta s / r}
        \qtext{where}
        g(\varepsilon) \equiv \frac{e^{-\varepsilon}}{(1-\varepsilon)^{1-\varepsilon}}.
    \end{equation*}
    Setting the left hand side to be at most $\delta$,
    we can express the number of samples $s$ in terms of $\varepsilon$ and $\delta$
    by taking logs, multiplying both sides by $-1$ to flip the logarithms, and rearranging:
    \begin{displaymath}
        r \left( g(\varepsilon) \right)^{\beta s / r} \leq \delta
        \quad\Rightarrow\quad
        s \geq \frac{r}{\beta} \frac{\ln(r/\delta)}{\ln(g(\varepsilon)^{-1})}
        .
    \end{displaymath}
    Observe that $\sigma_{\min}^2(\Mx{S} \UA) = \lambda_{\min}((\Mx{S}\UA)^{\trans} \Mx{S}\UA)$.
    Setting $\varepsilon = 1 - 1/\sqrt{2}$ yields $\sigma_{\min}^2(\Mx{S} \UA) \geq 1/\sqrt{2}$
    with failure probability $\delta$.
    Since $\varepsilon = 1 - 1/\sqrt{2}$,
    we have $C=1/\ln(g(\varepsilon)^{-1}) =
        \sqrt{2}/(\sqrt{2}-1-\log(2)/2) \approx 20.9080$  to complete the proof.
\end{proof}

\begin{remark}[Improvement in Constant]\label{rem:woodruff-compare}
    The original version of this article \cite{arXiv-LaKo22} utilized a version of \cref{lem:sc1} from Woodruff~\cite[Theorem 2.11]{Wo14} that had a much higher constant. This resulted in a substantially looser bound which had appeared overly pessimistic compared to experimental results.
    This impacted several downstream results in works such as \cite{LaKo22, hayashi2025randomized}. Here we briefly explore how we were able to improve the constant.

    Under the same conditions as \cref{lem:sc1}, Woodruff \cite[Theorem 2.11]{Wo14} states:
    \begin{displaymath}
        s > 144 \frac{r}{\beta \epsilon^2} \ln(2r/\delta)
        \quad \Rightarrow \quad
        \Prob\left\{ 1- \epsilon \leq \sigma_{\min}^2(\Mx{S} \UA) \leq 1 + \epsilon \right\} \geq 1-\delta
    \end{displaymath}
    For the purposes of satisfying the first structural condition, we only need one side of the bound, so we can replace $2r$ by $r$ inside the log term.
    Setting $\epsilon = 1 -1/\sqrt{2}$, Woodruff \cite[Theorem 2.11]{Wo14} then gives:
    \begin{displaymath}
        s > \left [ \frac{144}{(1-1/\sqrt{2})^2} \right ] \frac{r}{\beta} \ln(r/\delta)
        \quad \Rightarrow \quad
        \Prob\left\{ \sigma_{\min}^2(\Mx{S} \UA) \geq 1/\sqrt{2} \right\} \geq 1-\delta.
    \end{displaymath}
    Inside the proof of Woodruff~\cite[Theorem 2.11]{Wo14} which closely follows the proof of Magdon-Ismail~\cite[Corollary 16]{magdon2010row}, one can directly obtain the tighter sample bound:
    \begin{displaymath}
        s > \left [ 2 \left ( 1 + \frac{\epsilon}{3} \right ) \right ] \frac{r}{\beta \epsilon^2} \ln(2r/\delta)
        \quad \Rightarrow \quad
        \Prob\left\{ 1- \epsilon \leq \sigma_{\min}^2(\Mx{S} \UA) \leq 1 + \epsilon \right\} \geq 1-\delta
    \end{displaymath}
    Using the same logic as above, setting $\epsilon = 1 -1/\sqrt{2}$, and replacing $2r$ by $r$ inside the log term gives:
    \begin{displaymath}
        s > \left [ \frac{2}{(1-1/\sqrt{2})^2} + \frac{1}{(1-1/\sqrt{2})} \right ] \frac{r}{\beta} \ln(r/\delta)
        \quad \Rightarrow \quad
        \Prob\left\{ \sigma_{\min}^2(\Mx{S} \UA) \geq 1/\sqrt{2} \right\} \geq 1-\delta
    \end{displaymath}

    \Cref{tab:compare} compares the constants across these results.
    The constant in Woodruff~\cite[Theorem 2.11]{Wo14} can be improved from $\approx 1678.5870$ to $\approx 26.7279$ by directly analyzing the proof of the theorem rather than the statement.
    Nevertheless, the constant in \cref{lem:sc1} is still an improvement over this improved version of Woodruff~\cite[Theorem 2.11]{Wo14} because that results employs a matrix Chernoff bound that does not assume positive-semidefiniteness.

\end{remark}

\begin{table}
    \caption{Comparison of constants across results. All results give $s > C \frac{r}{\beta} \ln(r/\delta)$
        implies $\Prob\left\{ \sigma_{\min}^2(\Mx{S} \UA) \geq 1/\sqrt{2} \right\} \geq 1-\delta$ with the value of $C$ specified in the table and under the conditions specified in \cref{lem:sc1}.}
    \label{tab:compare}
    \begin{tabular}{lcc}
        \toprule
        Source                              & Constant                                                                                     & Relative Size \\
        \midrule
        \cite[Theorem 2.11]{Wo14} Statement & $C = \left [ \frac{144}{(1-1/\sqrt{2})^2} \right ] \approx 1678.5870$                        & $80.28$       \\
        \cite[Theorem 2.11]{Wo14} Proof     & $C = \left [ \frac{2}{(1-1/\sqrt{2})^2} + \frac{1}{(1-1/\sqrt{2})} \right ] \approx 26.7279$ & $1.28$        \\
        \Cref{lem:sc1}                      & $C = \left [ \frac{\sqrt{2}}{\sqrt{2}-1-\log(2)/2} \right ] \approx 20.9080$                 & $1.00$        \\
        \bottomrule
    \end{tabular}
\end{table}

\begin{remark}[Previous Appearances in the Literature]\label{rem:sc1-previous}
    \Cref{lem:sc1} is closely related to the fact that approximate leverage scores provide an $\epsilon$-accurate subspace embedding, namely the lower bound on $\sigma_{\min}^2(\Mx{S}\UA)$. An analogous result appears in Proposition 6.1.1 and the related Appendix A.3 of Murray et al.~\cite{murray2023randomized}, which itself traces back to Tropp~\cite[Problem 5.13]{tropp2020randomized}.
\end{remark}

\subsection{Second Structural Condition for Random Sampling}
\label{sec:sc2}

The second structural condition \cref{eq:sc2} can be proven using results for randomized matrix-matrix multiplication.
Consider the matrix product $\UA' \Bperp$.
This projects the part of  the columns of $\Mx{B}$ outside of the column space of $\A$ onto the column space
of $\A$ and thus by definition is equal to the all zeros  matrix $\Mx{0}_{r \times n}$ (we have assumed $\rank(\A) = r$).
This condition requires us to bound how well the sampled product
$\UA' \Mx{S}' \Mx{S} \Bperp$ approximates the original product.
We can do this via the following lemma from Drineas, Kannan, and Mahoney \cite{drineas2006fast}.

\begin{lemma}[Sketched Matrix Product \cite{drineas2006fast}]\label{lemma:randMatMult}
    Consider two matrices $\A \in \R^{N \times m}$ and $\Mx{B} \in \R^{N \times p}$.
    Let the sketching matrix
    $\Mx{S} \in \R^{s \times N} \sim \RandSample(s,\Vp)$
    where $s$ denotes the number of samples and
    $\V{p} \in [0,1]^N$ is a probability distribution such that
    $p_k \geq \beta \|\A(k,:)\|_2^2 / \|\A\|_F^2$ for all $k \in [N]$
    for some fixed $\beta > 0$.
    We then have the following guarantee on the quality of the approximate product:
    \begin{displaymath}
        \E \left [  \| \A' \Mx{B} - (\Mx{S} \A)^{\trans} \Mx{S} \Mx{B} \|_F^2 \right ] \leq \frac{1}{\beta s} \|\A\|_F^2 \|\Mx{B}\|_F^2.
    \end{displaymath}
\end{lemma}

\begin{proof}
    Fix $i, j$ to specify an element of the matrix product and let $\{\xi^{(1)}, \ldots, \xi^{(s)}\}$
    be the indices of the sampled rows of $\A$ (and $\Mx{B}$).
    We begin by calculating the expected value and variance of the corresponding
    element of the sampled matrix product,
    i.e., $\left[ (\Mx{S} \A) ^{\trans} \Mx{S} \Mx{B} \right ]_{ij}$.
    This can be written in terms of scalar random variables $X_t$ for $t = 1, \ldots, s$ as follows:
    \begin{equation*}
        X_t = \frac{\A(\xi^{(t)}, i)^{\trans} \Mx{B}(\xi^{(t)}, j)}{s p_{\xi^{(t)}}}
        \quad \Longrightarrow \quad
        \left[ (\Mx{S} \A) ^{\trans} \Mx{S} \Mx{B} \right ]_{ij} = \sum_{t = 1}^s X_t.
    \end{equation*}
    The expectation of $X_t$ and $X_t^2$ for all $t$ can be calculated as follows:
    \begin{align*}
        \E[X_t]   & = \sum_{k = 1}^N  p_k \frac{\A_{ki} \Mx{B}_{kj}}{s p_k} = \frac{1}{s} (\A' \Mx{B})_{ij},                                   \\
        \E[X_t^2] & = \sum_{k = 1}^N  p_k \frac{\A_{ki}^2 \Mx{B}_{kj}^2}{s^2 p_k^2} = \sum_{k = 1}^N  \frac{\A_{ki}^2 \Mx{B}_{kj}^2}{s^2 p_k}.
    \end{align*}
    The relation between $X_t$ and $[(\Mx{S} \A) ^{\trans} \Mx{S} \Mx{B}]_{ij}$
    gives $\E \left [[ (\Mx{S} \A) ^{\trans} \Mx{S} \Mx{B} ]_{ij} \right] = \sum_{t = 1}^s \E[X_t] = (\A' \Mx{B})_{ij}$
    and hence the estimator is unbiased.
    Furthermore, since the estimated matrix element is the sum of $s$ independent random variables,
    its variance can be calculated as follows:
    \begin{align*}
        \Var \left [ [(\Mx{S} \A) ^{\trans} \Mx{S} \Mx{B}]_{ij} \right ] & = \Var \left [\sum_{t = 1}^s X_t  \right ] = \sum_{t = 1}^s \Var[X_t]
        = \sum_{t = 1}^s \left ( \E[X_t^2] - \E[X_t]^2 \right )                                                                                                                                              \\
                                                                         & = \sum_{t = 1}^s \left ( \sum_{k = 1}^N  p_k \frac{\A_{ki}^2 \Mx{B}_{kj}^2}{s^2 p_k^2} - \frac{1}{s^2} (\A' \Mx{B})_{ij} \right )
        = \sum_{k = 1}^N \frac{\A_{ki}^2 \Mx{B}_{kj}^2}{s p_k} - \frac{1}{s} (\A' \Mx{B})_{ij}.
    \end{align*}
    Now we turn to the expectation we want to bound and apply these results:
    \begin{align*}
        \E \left [  \| \A \Mx{B} - (\Mx{S} \A)^{\trans} \Mx{S} \Mx{B} \|_F^2 \right ]
         & = \sum_{i = 1}^m \sum_{j = 1}^p \E \left [ \left ( [(\Mx{S} \A) ^{\trans} \Mx{S} \Mx{B}]_{ij} - (\A' \Mx{B})_{ij} \right )^2 \right ]                                                     \\
         & = \sum_{i = 1}^m \sum_{j = 1}^p \E \left [ \left ( [(\Mx{S} \A) ^{\trans} \Mx{S} \Mx{B}]_{ij} - \E \big [[(\Mx{S} \A) ^{\trans} \Mx{S} \Mx{B}]_{ij} \big ] \right )^2 \right ]            \\
         & = \sum_{i = 1}^m \sum_{j = 1}^p \Var \left [ \big[(\Mx{S} \A) ^{\trans} \Mx{S} \Mx{B} \big]_{ij} \right ]
        = \sum_{i = 1}^m \sum_{j = 1}^p \left (\sum_{k = 1}^N \frac{\A_{ki}^2 \Mx{B}_{kj}^2}{s p_k} - \frac{1}{s} (\A' \Mx{B})_{ij} \right )                                                         \\
         & =   \sum_{k = 1}^N \frac{\left ( \sum_{i = 1}^m \A_{ki}^2 \right ) \left ( \sum_{j = 1}^p \Mx{B}_{kj}^2 \right )}{s p_k} -   \frac{1}{s} \sum_{i = 1}^m \sum_{j = 1}^p  (\A' \Mx{B})_{ij} \\
         & = \frac{1}{s} \sum_{k = 1}^N \frac{ \| \A(k,:)\|_2^2 \| \Mx{B}(k, :)\|_2^2}{p_k} -   \frac{1}{s} \| \A' \Mx{B} \|_F^2                                                                     \\
         & \leq \frac{1}{s} \sum_{k = 1}^N \frac{ \| \A(k,:)\|_2^2 \| \Mx{B}(k, :)\|_2^2}{p_k},
    \end{align*}
    where in the last line we have used that the Frobenius norm of a matrix is strictly positive.
    Lastly, we use our assumption on the probabilities $p_k \geq \frac{\beta \|\A(k,:)\|_2^2}{\|\A\|_F^2}$ to obtain the desired bound:
    \begin{align*}
        \E \left [  \| \A' \Mx{B} - (\Mx{S} \A)^{\trans} \Mx{S} \Mx{B} \|_F^2 \right ]
         & \leq \frac{1}{s} \sum_{k = 1}^N \frac{ \| \A(k,:)\|_2^2 \| \Mx{B}(k, :)\|_2^2}{p_k},                                                 \\
         & \leq \frac{1}{s} \sum_{k = 1}^N \left ( \|\A'\|_F^2 \frac{ \| \A(k,:)\|_2^2 \| \Mx{B}(k, :)\|_2^2 }{\beta \|\A(k,:)\|_2^2} \right ), \\
         & = \frac{1}{\beta s} \|\A\|_F^2 \sum_{k = 1}^N \| \Mx{B}(k, :)\|_2^2 = \frac{1}{\beta s} \|\A\|_F^2 \|\Mx{B}\|_F^2.
    \end{align*}
\end{proof}

\noindent We can apply \cref{lemma:randMatMult} to bound the probability of \cref{eq:sc2} holding.

\begin{lemma}[SC2 for Random Sampling]\label{lem:sc2}
    For the overdetermined least squares problem \cref{eq:lsq-sketched},
    assume $\A \in \R^{N \times r}$ is full rank, its SVD is $\UA \Mx{\Sigma}_{\A} \Mx{V}{\A}'$, and its row leverage scores are $\ell_i(\A)$.
    Define the probability distribution $\V{p} \in [0,1]^N$ and assume there exists $\beta \in (0,1]$
    such that $p_i \geq \beta \ell_i(\A)/r$ for all $i \in [N]$.
    Construct row sampling and rescaling matrix $\Mx{S} \in \R^{s \times N} \sim \RandSample(s,\Vp)$.
    Then
    \begin{displaymath}
        s \geq  \frac{ 2 r}{\beta \delta \epsilon }
        \quad \Rightarrow \quad
        \Prob \left\{
        \|\UA' \Mx{S}' \Mx{S} \Bperp \|_F^2 \leq \epsilon \resid^2/2
        \right\} \geq 1 - \delta.
    \end{displaymath}
\end{lemma}
\begin{proof}
    Apply \cref{lemma:randMatMult} to obtain a bound on the expected value:
    \begin{align*}
        \E \left [  \| \UA' \Mx{S}' \Mx{S} \Bperp \|_F^2 \right ]
         & = \E \left [  \| \Mx{0}_{r \times n} - \UA' \Mx{S}' \Mx{S} \Bperp \|_F^2 \right ]
        = \E \left [  \| \UA' \Bperp - \UA' \Mx{S}' \Mx{S} \Bperp \|_F^2 \right ]               \\
         & \leq \frac{1}{\beta s} \|\UA\|_F^2 \|\Bperp\|_F^2 = \frac{r}{\beta s} \|\Bperp\|_F^2
        \mbox{}=\frac{r}{\beta s} \mathcal{R}^2.
    \end{align*}
    Markov's inequality states that for non-negative random variable $X$ and scalar $t > 0$, we can bound the probability that $X \geq t$ as $\Prob [ X \geq t] \leq \E[X]/t$.
    We can apply this inequality to bound the probability that the sketching matrix violates \cref{eq:sc2}:
    \begin{align*}
        \Prob
        \left\{  \| \UA' \Mx{S}' \Mx{S} \Bperp \|_F^2  \geq \frac{\epsilon \|\Bperp\|_F^2}{2} \right\} & \leq \frac{ 2 \E \left [  \| \UA' \Mx{S}' \Mx{S} \Mx{\Bperp} \|_F^2 \right ]}{\epsilon \|\Bperp\|_F^2} \leq \frac{ 2 r}{\beta \epsilon s}
    \end{align*}
    where in the last step we have used our bound for the expected value. Thus if we set the right-hand side equal to $\delta$, we obtain that the probability that \cref{eq:sc2} holds is greater than or equal to $1 - \delta$ as desired.  Solving for $s$ yields that we thus must have $s \geq  \frac{ 2 r}{\beta \delta \epsilon }$.
\end{proof}

\subsection{Main Theorem for Random Sampling}
\label{sec:main-theorem}

We combine the above results to obtain the theorem below.
Note that this theorem is identical
to Larsen and Kolda~\cite[Theorem 6]{LaKo22} except that the constants are improved.

\begin{theorem}[Random Sampling]%
    \label{thm:sketching-app}
    Consider the least squares problem
    $\min_{\Mx{X} \in \Real^{r \times n}} \| \A \Mx{X} - \B\|^2$
    where $\A \in \Real^{N \times r}$ with $r \ll N$ and $\rank(\A) = r$
    and  $\B \in \Real^{N \times n}$.
    Let $\Vp \in [0,1]^N$ be a probability distribution and
    assume there exists a fixed
    $\beta \in (0,1]$ such that
    \begin{displaymath}
        p_i \geq \beta \frac{\ell_i(\A)}{r} \qtext{for all} i \in [N].
    \end{displaymath}
    For any  $\epsilon, \delta \in (0,1)$,
    set
    \begin{displaymath}
        s=({r}/{\beta}) \max \set{ C \log (2r/\delta), {4}/(\delta \epsilon) }
        \qtext{where}
        C=\sqrt{2}/(\sqrt{2}-1-\log(2)/2) \approx 20.9080   ,
    \end{displaymath}
    and let $\Mx{S} = \RandSample(s,\Vp)$.
    Define
    $\Xopt \equiv \arg \min_{\Mx{X} \in \Real^{r \times n}} \| \A \Mx{X} - \B\|^2$.
    Then $\Xtildeopt \equiv \arg \min_{\Mx{X} \in \Real^{r \times n}} \| \Mx{S} \A \Mx{X} - \Mx{S} \B\|_F^2$
    satisfies $\| \A \Xtildeopt - \B \|_F^2 \leq (1+\epsilon) \| \A \Xopt - \B \|_F^2$
    with probability at least $1-\delta$.
\end{theorem}

\begin{proof}
    Applying \cref{lem:sc1}, we have that \cref{eq:sc1} holds with probability $1-\delta/2$ if
    $s=C r \log (2r/\delta)/\beta$.
    Applying \cref{lem:sc2}, we have that \cref{eq:sc2} holds with probability $1-\delta/2$ if
    $s=4r/(\beta \delta \epsilon)$.
    Hence, a union bound says that \cref{eq:sc1,eq:sc2} both hold with probability $1-\delta$ if
    $s= (r/\beta) \max\set{ C \log(2r/\delta), 4/(\delta\epsilon)}$.
    Combining this with \cref{thm:sketch-bound} yields the result.
\end{proof}

\section{Hybrid Deterministic and Random Sampling}
\label{sec:hybrid-sampling}

When sampling probabilities are concentrated, i.e., a small subset of rows accounts for a significant portion of the probability mass, standard random sampling repeatedly selects the same high-probability rows.
This leads to inefficient sketches.
We review the hybrid approach proposed by Larsen and Kolda~\cite{LaKo22} that deterministically includes high-probability rows and randomly samples the remainder.
Then we extend the theory of Hayashi et al.~\cite{hayashi2025randomized} to describe
how this reduces the sample complexity as compared to standard
approximate leverage score sampling.
We use \cref{thm:sketch-bound} as we did for the random sampling case,
and the major difference is in showing that the two structural conditions hold for hybrid sampling. The final result is in \cref{sec:hybrid-main-theorem}, followed by a comparison to the standard sampling case in \cref{rem:hybrid-comparison}.

\subsection{Hybrid Sampling}
\label{sec:hybrid-sampling-def}

We first define the hybrid sampling procedure. The essential idea is that a few rows are included deterministically, while the rest are sampled randomly according to a rescaled probability distribution.
We will later see that the highest-probability rows should be chosen for deterministic inclusion.

\begin{definition}[Hybrid Sampling \cite{LaKo22}]\label{def:hybridsample}
    Let $\DetSet \subset [N]$ be an ordered set of indices to include deterministically.
    Define $\pdet = \sum_{i \in \DetSet} p_i$.
    We say $\Mx{S} \in \Real^{s \times N} \sim \HybridSample(s, \Vp, \DetSet)$ if the sampling matrix has the form
    \begin{displaymath}
        \Mx{S} = \begin{bmatrix} \Sdet \\ \Srnd \end{bmatrix}
    \end{displaymath}
    where $\Sdet \in \Real^{d \times N}$ picks out the $d \equiv |\DetSet|$ deterministic rows and
    $\Srnd \in \Real^{\tilde{s} \times N}$ randomly samples and reweights $\tilde{s} \equiv s - |\DetSet|$ rows.

    The matrix $\Sdet$ is constructed as follows.
    For each $j \in [d]$, let
    $k_j$ be the $j$th element of $\DetSet$.
    Then the $j$th row of $\Sdet$ is the $(k_j)$-unit vector in $\Real^N$, i.e., a vector of all zeros except for a 1 in position $k_j$.

    For the matrix $\Srnd$, let $\xi^{(j)} \sim \Multi(\bm{\tilde{p}})$ for each $j \in [\tilde{s}]$ where $\bm{\tilde{p}}$ is the rescaled probability distribution over $[N] \setminus \DetSet$ such that
    \begin{displaymath}
        \tilde{p}_i =
        \begin{cases}
            \frac{p_i}{1-\pdet} & \text{if } i \in [N] \setminus \DetSet, \\
            0                   & \text{if } i \in \DetSet.
        \end{cases}
    \end{displaymath}
    Then we define the entries of $\Srnd$ as
    \begin{displaymath}
        \Srnd(j,i) =
        \begin{cases}
            \sqrt{\frac{1}{\tilde{s} \, \tilde{p}_i}} & \text{if } \xi^{(j)} = i, \\
            0                                         & \text{otherwise},
        \end{cases}
        \qtext{for all} (j,i) \in [\tilde{s}] \times [N].
    \end{displaymath}
\end{definition}

We can see that this hybrid sampling is unbiased in the following lemma.

\begin{lemma}\label{lem:hybrid-expectation}
    Let $\V{x} \in \Real^N$ and let $\Vp \in [0,1]^N$ be a probability distribution such that $p_i > 0$ if $x_i \neq 0$.
    Let $\DetSet \subset [N]$ be any subset of indices.
    Let $\Mx{S} \sim \HybridSample(s, \Vp, \DetSet)$.
    Then the hybrid sampling is unbiased, i.e.,
    \begin{displaymath}
        \Exp{ \| \Mx{S} \V{x} \|_2^2 } = \|\V{x}\|_2^2.
    \end{displaymath}
\end{lemma}

\begin{proof}
    Split the norm into deterministic and random parts:
    \begin{displaymath}
        \| \Mx{S} \V{x} \|_2^2 = \| \Sdet \V{x} \|_2^2 + \| \Srnd \V{x} \|_2^2.
    \end{displaymath}
    For the deterministic part, $\| \Sdet \V{x} \|_2^2 = \sum_{i \in \DetSet} x_i^2$.
    For the random part, each sample $\xi^{(j)}$ selects index $i \in [N] \setminus \DetSet$ with probability $p_i / (1-\pdet)$:
    \begin{displaymath}
        \Exp{ \| \Srnd \V{x} \|_2^2 }
        = \Exp{ \sum_{j=1}^{\tilde{s}} \frac{x_{\xi^{(j)}}^2}{\tilde{s} \tilde{p}_{\xi^{(j)}}} }
        = \sum_{j=1}^{\tilde{s}} \sum_{i \notin \DetSet} \frac{p_i}{1-\pdet} \cdot \frac{1-\pdet}{\tilde{s} \, p_i}  x_i^2
        = \sum_{j=1}^{\tilde{s}} \frac{1}{\tilde{s}} \sum_{i \notin \DetSet} x_i^2 = \sum_{i \notin \DetSet} x_i^2.
    \end{displaymath}
    Combining both parts gives the result.
\end{proof}

\subsection{First Structural Condition for Hybrid Sampling}
\label{sec:hybrid-sc1}

We begin with a hybrid version of \cref{lem:sc1}. Comparing the two theorems, we see that a factor of $r/\beta$ in the sample complexity is replaced by $(1-\pdet) r/\beta + 1$. In other words, the sample complexity is reduced by a factor of roughly $(1-\pdet)$.

\begin{lemma}[Hybrid SC1]\label{lem:hybrid-sc1}
    Consider full rank $\A \in \R^{N \times r}$, its SVD $\UA \Mx{\Sigma}_{\A} \Mx{V}{\A}'$, and row leverage scores $\ell_i(\A)$.
    Let $\DetSet \subset [N]$ be a deterministic index set.
    Let $\Vp$ be a probability distribution such that $p_i \geq \beta \ell_i(\A)/r$ for all $i \in [N]$ and some $\beta \in (0,1]$.
    Construct the hybrid sampling matrix $\Mx{S} \sim \HybridSample(s, \Vp, \DetSet)$.
    Then,
    \begin{displaymath}
        \tilde{s} \geq C\left((1-\pdet) \frac{r}{\beta} + 1\right)\ln(r/\delta)
        \quad \Rightarrow \quad
        \Pr\left\{
        \sigma_{\min}^2(\Mx{S}\UA) \geq 1/\sqrt{2}
        \right\} \geq 1-\delta,
    \end{displaymath}
    where $C = \sqrt{2}/(\sqrt{2}-1-\log(2)/2) \approx 20.9080$.
\end{lemma}

\begin{proof}
    Define $\Mx[\tilde]{U}_{\A}$ to be a copy of $\UA$ with rows indexed by $\DetSet$ zeroed out and $\Mx{U}{\det}$ to be another copy but with the rows not in $\DetSet$ zeroed out.
    These two matrices thus decompose $\Mx{U}_{\A}$ into a sum of two matrices of the same size but with non-overlapping non-zero rows, i.e. $\Mx{U}_{\A} = \Mx{U}{\det} + \Mx[\tilde]{U}_{\A}$.
    The hybrid sketch then admits a similar decomposition and satisfies
    \begin{displaymath}
        (\Mx{S}\UA)^{\trans}\Mx{S}\UA = \Mx{U}{\det}'\Mx{U}{\det} + \Mx[\tilde]{U}{\A}'\Srnd'\Srnd\Mx[\tilde]{U}_{\A},
    \end{displaymath}

    Define the random matrices
    \begin{displaymath}
        \Mx{W}_k = \frac{1}{\tilde{s}\tilde{p}_{\xi^{(k)}}} \Mx[\tilde]{U}_{\A}(\xi^{(k)},:)^{\trans}\Mx[\tilde]{U}_{\A}(\xi^{(k)},:) \in \R^{r \times r},
    \end{displaymath}
    where $\xi^{(k)}$ is sampled according to $\tilde{p}_i = p_i/(1-\pdet)$ for $i \notin \DetSet$.
    Then $\sum_{k=1}^{\tilde{s}} \Mx{W}_k = \Mx[\tilde]{U}{\A}'\Srnd'\Srnd\Mx[\tilde]{U}_{\A}$ and
    \begin{displaymath}
        \sum_{k=1}^{\tilde{s}} \E[\Mx{W}_k] = \sum_{k=1}^{\tilde{s}} \sum_{i \notin \DetSet} \tilde{p}_i \cdot \frac{1}{\tilde{s}\tilde{p}_i} \Mx[\tilde]{U}_{\A}(i,:)^{\trans}\Mx[\tilde]{U}_{\A}(i,:) = \Mx[\tilde]{U}{\A}'\Mx[\tilde]{U}_{\A}. %
    \end{displaymath}
    By construction, $\Mx[\tilde]{U}_{\A}$ and $\Mx{U}{\det}$ have non-overlapping non-zero rows and thus $\Mx[\tilde]{U}_{\A}^{\trans}\Mx{U}{\det} = \Mx{U}{\det}^{\trans}\Mx[\tilde]{U}_{\A} = \Mx{0}$. As $\Mx{U}_{\A}$ is orthonormal this means:
    \begin{displaymath}
        \Mx{I}_r = \Mx{U}_{\A}^{\trans}\Mx{U}_{\A} = (\Mx{U}{\det} + \Mx[\tilde]{U}_{\A})^{\trans}(\Mx{U}{\det} + \Mx[\tilde]{U}_{\A}) = \Mx{U}{\det}'\Mx{U}{\det} + \Mx[\tilde]{U}{\A}'\Mx[\tilde]{U}_{\A},
    \end{displaymath}
    and $\sum_{k=1}^{\tilde{s}} \E[\Mx{W}_k] = \Mx[\tilde]{U}{\A}'\Mx[\tilde]{U}_{\A} = \Mx{I}_r - \Mx{U}{\det}'\Mx{U}{\det}$.

    Now define the augmented random matrices that include the deterministic contribution:
    \begin{displaymath}
        \Mx{X}_k = \Mx{W}_k + \frac{1}{\tilde{s}}\Mx{U}{\det}'\Mx{U}{\det} \qtext{for} k = 1,\ldots,\tilde{s}.
    \end{displaymath}
    Then $\sum_{k=1}^{\tilde{s}} \Mx{X}_k = \Mx{U}{\det}'\Mx{U}{\det} + \Mx[\tilde]{U}{\A}'\Srnd'\Srnd\Mx[\tilde]{U}_{\A} = (\Mx{S}\UA)^{\trans}\Mx{S}\UA$ and:
    \begin{displaymath}
        \sum_{k=1}^{\tilde{s}} \E[\Mx{X}_k] = \left [ \sum_{k=1}^{\tilde{s}} \E[\Mx{W}_k] \right ] + \Mx{U}{\det}'\Mx{U}{\det} = \Mx{I}_r - \Mx{U}{\det}'\Mx{U}{\det} + \Mx{U}{\det}'\Mx{U}{\det} = \Mx{I}_r.
    \end{displaymath}

    We can thus apply the matrix Chernoff bound (\cref{thm:matrixChernof}) provided we bound the maximum eigenvalue. For $\Mx{W}_k$ we have:
    \begin{displaymath}
        \lambda_{\max}(\Mx{W}_k) = \frac{\ell_{\xi^{(k)}}(\A)}{\tilde{s}\tilde{p}_{\xi^{(k)}}}.
    \end{displaymath}
    Since $\tilde{p}_i = p_i/(1-\pdet) \geq \beta\ell_i(\A)/(r(1-\pdet))$, we have
    \begin{displaymath}
        \lambda_{\max}(\Mx{W}_k) \leq \frac{r(1-\pdet)}{\beta\tilde{s}}.
    \end{displaymath}
    By Weyl's inequality~{\cite[Theorem~4.3.1]{HoJo85}}, $\lambda_{\max}(\Mx{A}+\Mx{B}) \leq \lambda_{\max}(\Mx{A}) + \lambda_{\max}(\Mx{B})$ for symmetric matrices, so
    \begin{displaymath}
        \lambda_{\max}(\Mx{X}_k) \leq \lambda_{\max}(\Mx{W_k}) + \lambda_{\max}\left(\frac{1}{\tilde{s}}\Mx{U}{\det}'\Mx{U}{\det}\right) \leq \frac{r(1-\pdet)}{\beta\tilde{s}} + \frac{1}{\tilde{s}} \lambda_{\max}\left(\Mx{U}{\det}'\Mx{U}{\det}\right).
    \end{displaymath}
    From before, we have that $\Mx{U}{\det}'\Mx{U}{\det} = \Mx{I}_r - \Mx[\tilde]{U}{\A}'\Mx[\tilde]{U}_{\A}$ and thus every eigenvector of $\Mx[\tilde]{U}{\A}'\Mx[\tilde]{U}_{\A}$ with eigenvalue $\lambda_i(\Mx[\tilde]{U}{\A}'\Mx[\tilde]{U}_{\A})$ will also be an eigenvector of $\Mx{U}{\det}'\Mx{U}{\det}$ with eigenvalue $1 - \lambda_i(\Mx[\tilde]{U}{\A}'\Mx[\tilde]{U}_{\A})$. Since both $\Mx[\tilde]{U}{\A}'\Mx[\tilde]{U}_{\A}$ and $\Mx{U}{\det}'\Mx{U}{\det}$ are PSD, this implies $\lambda_i(\Mx{U}{\det}'\Mx{U}{\det}) \in [0,1]$ for all $i \in [r]$, and in particular $\lambda_{\max}(\Mx{U}{\det}'\Mx{U}{\det}) \leq 1$:
    \begin{displaymath}
        \lambda_{\max}(\Mx{X}_k) \leq \frac{r(1-\pdet)}{\beta\tilde{s}} + \frac{1}{\tilde{s}} = \frac{r(1-\pdet) + \beta}{\beta\tilde{s}},
    \end{displaymath}
    so $L = (r(1-\pdet) + \beta)/(\beta\tilde{s})$.

    Applying the matrix Chernoff bound with $\varepsilon = 1 - 1/\sqrt{2}$:
    \begin{displaymath}
        \Prob\left\{\lambda_{\min}((\Mx{S}\UA)^{\trans}\Mx{S}\UA) \leq 1/\sqrt{2}\right\}
        \leq r \cdot g(\varepsilon)^{\mu_{\min}/L}
        = r \cdot g(\varepsilon)^{\beta\tilde{s}/(r(1-\pdet)+\beta)}.
    \end{displaymath}
    Setting this equal to $\delta$ and solving:
    \begin{displaymath}
        \tilde{s} \geq \frac{r(1-\pdet)+\beta}{\beta} \cdot \frac{\ln(r/\delta)}{\ln(g(\varepsilon)^{-1})} = C\left(\frac{r(1-\pdet)}{\beta} + 1\right) \ln(r/\delta),
    \end{displaymath}
    where $C = 1/\ln(g(\varepsilon)^{-1}) \approx 20.9080$.
\end{proof}

We can further simplify the sample complexity if we assume that $r(1-\pdet) \geq \beta$.
This is a reasonable assumption. Generally, $r$ is somewhat large and $\pdet$ is not that close to 1, so that $r(1-\pdet)$ is significantly larger than 1. Since $\beta \leq 1$, the assumption follows easily. The simplification is a gross overestimate but makes the expression easier to interpret in the final result.

\begin{corollary}[SC1 Simplification]\label{cor:hybrid-sc1}
    Under the assumptions of \cref{lem:hybrid-sc1}, if we assume $r(1-\pdet) \geq \beta$, then we have
    \begin{displaymath}
        \tilde{s} \geq 2 C (1-\pdet) \frac{r}{\beta} \ln(r/\delta)
        \quad \Rightarrow \quad
        \Pr\left\{
        \sigma_{\min}^2(\Mx{S}\UA) \geq 1/\sqrt{2}
        \right\} \geq 1-\delta.
    \end{displaymath}
\end{corollary}

\subsection{Second Structural Condition for Hybrid Sampling}
\label{sec:hybrid-sc2}

Next, we present a hybrid version of \cref{lemma:randMatMult} that accounts for the deterministic and random parts separately
and shows that the result depends only on the random part.
Note, however, that there is a different assumption on the sampling probabilities relating to only the rows not in the deterministic set.

\begin{lemma}[Hybrid Random Matrix Multiplication]\label{lemma:hybridRandMatMult}
    Consider two matrices $\A \in \R^{N \times m}$ and $\Mx{B} \in \R^{N \times p}$.
    Let $\DetSet \subset [N]$ be a set of indices to be included deterministically.
    Let $\Mx{S} \sim \HybridSample(s, \Vp, \DetSet)$.
    Define $\Mx[\tilde]{A}$ and $\Mx[\tilde]{B}$ to zero out the rows of $\A$ and $\Mx{B}$ indexed by $\DetSet$.
    Assume there exists $\tilde{\beta} > 0$ such that  $\tilde{p}_i \geq \tilde{\beta} \|\Mx[\tilde]{A}(i,:)\|_2^2 / \|\Mx[\tilde]{A}\|_F^2$ for all $i \in [N] \setminus \DetSet$.
    Then, we have
    \begin{displaymath}
        \E \left[ \| \A' \Mx{B} - (\Mx{S} \A)^{\trans} \Mx{S} \Mx{B} \|_F^2 \right] \leq \frac{1}{\tilde{\beta} \tilde{s}} \|\Mx[\tilde]{A}\|_F^2 \|\Mx[\tilde]{B}\|_F^2.
    \end{displaymath}
\end{lemma}

\begin{proof}
    From \cref{def:hybridsample}, we have
    \begin{displaymath}
        (\Mx{S} \A)^{\trans} \Mx{S} \Mx{B}
        = (\Sdet \A)^{\trans} \Sdet \Mx{B} + (\Srnd \Mx[\tilde]{A})^{\trans} \Srnd \Mx[\tilde]{B}
        = \sum_{i \in \DetSet} \A(i,:)^{\trans} \Mx{B}(i,:) + \sum_{t=1}^{\tilde{s}} \frac{1}{\tilde{s} \tilde{p}_{\xi^{(t)}}} \Mx[\tilde]{A}(\xi^{(t)},:)^{\trans} \Mx[\tilde]{B}(\xi^{(t)},:),
    \end{displaymath}
    where $\{\xi^{(1)}, \ldots, \xi^{(\tilde{s})}\}$ are drawn according to the rescaled distribution $\tilde{p}_i = p_i/(1-\pdet)$ for $i \in [N] \setminus \DetSet$. The definitions of $\Mx[\tilde]{A}$ and $\Mx[\tilde]{B}$ mean that $\Srnd \A = \Srnd \Mx[\tilde]{A}$ and $\Srnd \B = \Srnd \Mx[\tilde]{B}$.

    Because the deterministic part exactly recovers the corresponding rows of $\A' \Mx{B}$, we have
    \begin{displaymath}
        \A' \Mx{B} - (\Mx{S} \A)^{\trans} \Mx{S} \Mx{B} = \Mx[\tilde]{A}' \Mx[\tilde]{B} - (\Srnd \Mx[\tilde]{A})^{\trans} \Srnd \Mx[\tilde]{B},
    \end{displaymath}
    and it suffices to bound the expectation involving only the random part.
    Following the proof of \cref{lemma:randMatMult}, fix indices $i, j$ and define
    $X_t = \frac{1}{\tilde{s} \tilde{p}_{\xi^{(t)}}} \Mx[\tilde]{A}(\xi^{(t)}, i) \Mx[\tilde]{B}(\xi^{(t)}, j)$.
    The expectation is
    \begin{displaymath}
        \E[X_t] = \sum_{k \notin \DetSet} \tilde{p}_k \frac{1}{\tilde{s} \tilde{p}_k} \Mx[\tilde]{A}(k,i) \Mx[\tilde]{B}(k,j) = \frac{1}{\tilde{s}} (\Mx[\tilde]{A}' \Mx[\tilde]{B})_{ij},
    \end{displaymath}
    confirming the estimator is unbiased.

    The variance calculation yields:
    \begin{displaymath}
        \Var\left[\sum_{t=1}^{\tilde{s}} X_t\right] = \sum_{k \notin \DetSet} \frac{\Mx[\tilde]{A}(k,i)^2 \Mx[\tilde]{B}(k,j)^2}{\tilde{s} \tilde{p}_k} - \frac{1}{\tilde{s}}(\Mx[\tilde]{A}' \Mx[\tilde]{B})_{ij}^2.
    \end{displaymath}
    Summing over all $i,j$:
    \begin{align*}
        \E\left[\|\Mx[\tilde]{A}' \Mx[\tilde]{B} - (\Srnd \Mx[\tilde]{A})^{\trans} \Srnd \Mx[\tilde]{B}\|_F^2\right]
         & = \sum_{i=1}^m \sum_{j=1}^p \Var\left[\sum_{t=1}^{\tilde{s}} X_t\right]                                                                                                                \\
         & = \frac{1}{\tilde{s}} \sum_{k \notin \DetSet} \frac{\|\Mx[\tilde]{A}(k,:)\|_2^2 \|\Mx[\tilde]{B}(k,:)\|_2^2}{\tilde{p}_k} - \frac{1}{\tilde{s}} \|\Mx[\tilde]{A}' \Mx[\tilde]{B}\|_F^2 \\
         & \leq \frac{1}{\tilde{s}} \sum_{k \notin \DetSet} \frac{\|\Mx[\tilde]{A}(k,:)\|_2^2 \|\Mx[\tilde]{B}(k,:)\|_2^2}{\tilde{p}_k}.
    \end{align*}
    Now applying the assumption $\tilde{p}_k \geq \tilde{\beta} \|\Mx[\tilde]{A}(k,:)\|_2^2/\|\Mx[\tilde]{A}\|_F^2$:
    \begin{align*}
        \E\left[\|\Mx[\tilde]{A}' \Mx[\tilde]{B} - (\Srnd \Mx[\tilde]{A})^{\trans} \Srnd \Mx[\tilde]{B}\|_F^2\right]
         & \leq \frac{1}{\tilde{s}} \sum_{k \notin \DetSet} \frac{\|\Mx[\tilde]{A}\|_F^2}{\tilde{\beta}} \|\Mx[\tilde]{B}(k,:)\|_2^2 \\
         & = \frac{\|\Mx[\tilde]{A}\|_F^2}{\tilde{\beta} \tilde{s}} \sum_{k \notin \DetSet} \|\Mx[\tilde]{B}(k,:)\|_2^2
        = \frac{\|\Mx[\tilde]{A}\|_F^2 \|\Mx[\tilde]{B}\|_F^2}{\tilde{\beta} \tilde{s}}.
    \end{align*}
\end{proof}

We now apply \cref{lemma:hybridRandMatMult} to prove the hybrid version of \cref{lem:sc2}.
Observe here that we have a modified condition on the sampling probabilities. The quantity $\nu \leq r$ represents the remaining leverage score mass not covered by the deterministic set.
We relate this back to our original probability assumption in \cref{cor:hybrid-sc2}.

\begin{lemma}[Hybrid SC2]\label{lem:hybrid-sc2}
    Consider full rank $\A \in \R^{N \times r}$, its SVD $\UA \Mx{\Sigma}_{\A} \Mx{V}{\A}'$, and row leverage scores $\ell_i(\A)$.
    Let $\DetSet \subset [N]$ be a deterministic index set.
    Define $\nu = \sum_{i \notin \DetSet} \ell_i(\A)$.
    Construct the hybrid sampling matrix $\Mx{S} \sim \HybridSample(s, \Vp, \DetSet)$.
    Let $\pdet = \sum_{i \in \DetSet} p_i$ and define the rescaled probabilities $\tilde{p}_i = p_i/(1-\pdet)$ for $i \notin \DetSet$.
    Assume there exists $\tilde{\beta} \in (0,1]$ such that $\tilde{p}_i \geq \tilde{\beta} \ell_i(\A)/\nu$ for all $i \notin \DetSet$.
    Then,
    \begin{displaymath}
        \tilde{s} \geq \frac{2\nu}{\tilde{\beta} \delta \epsilon}
        \quad \Rightarrow \quad
        \Pr\left\{ \| \UA' \Mx{S}' \Mx{S} \Bperp \|_F^2 \leq \frac{\epsilon \resid^2}{2} \right\} \geq 1-\delta.
    \end{displaymath}
\end{lemma}

\begin{proof}
    Define $\Mx[\tilde]{U}_{\A}$ and $\Mx[\tilde]{B}'$ to be copies of $\UA$ and $\Bperp$ with rows indexed by $\DetSet$ zeroed out (so all indexing remains consistent with the original matrices).
    With hybrid sampling, the sketch matrix satisfies $\Mx{S}' \Mx{S} = \Sdet' \Sdet + \Srnd' \Srnd$,
    where $\Sdet' \Sdet$ is a diagonal matrix with ones at positions in $\DetSet$.

    Since $\UA' \Bperp = \Mx{0}_{r \times n}$, we can decompose:
    \begin{align*}
        \UA' \Mx{S}' \Mx{S} \Bperp
         & = \UA' \Sdet' \Sdet \Bperp + \UA' \Srnd' \Srnd \Bperp                                                    \\
         & = \sum_{i \in \DetSet} \UA(i,:)^{\trans} \Bperp(i,:) + \Mx[\tilde]{U}{\A}' \Srnd' \Srnd \Mx[\tilde]{B}'.
    \end{align*}

    Using $\UA' \Bperp = \Mx{0}$, we have $\sum_{i \in \DetSet} \UA(i,:)^{\trans} \Bperp(i,:) = -\Mx[\tilde]{U}{\A}' \Mx[\tilde]{B}'$.
    Therefore:
    \begin{displaymath}
        \UA' \Mx{S}' \Mx{S} \Bperp = \Mx[\tilde]{U}{\A}' \Srnd' \Srnd \Mx[\tilde]{B}' - \Mx[\tilde]{U}{\A}' \Mx[\tilde]{B}'.
    \end{displaymath}

    Apply \cref{lemma:hybridRandMatMult} with $\A = \Mx[\tilde]{U}_{\A}$ and $\Mx{B} = \Mx[\tilde]{B}'$.
    Note that $\|\Mx[\tilde]{U}_{\A}\|_F^2 = \sum_{i \notin \DetSet} \ell_i(\A) = \nu$ and $\|\Mx[\tilde]{B}\|_F^2 \leq \|\Bperp\|_F^2 = \resid^2$.
    Since $\ell_i(\A) = \|\Mx[\tilde]{U}_{\A}(i,:)\|_2^2$ for $i \notin \DetSet$, the assumption $\tilde{p}_i \geq \tilde{\beta} \ell_i(\A)/\nu$ is equivalent to $\tilde{p}_i \geq \tilde{\beta} \|\Mx[\tilde]{U}_{\A}(i,:)\|_2^2/\|\Mx[\tilde]{U}_{\A}\|_F^2$.
    Thus:
    \begin{displaymath}
        \E\left[\|\UA' \Mx{S}' \Mx{S} \Bperp\|_F^2\right] \leq \frac{\nu}{\tilde{\beta} \tilde{s}} \|\Mx[\tilde]{B}\|_F^2 \leq \frac{\nu \resid^2}{\tilde{\beta} \tilde{s}}.
    \end{displaymath}

    Applying Markov's inequality:
    \begin{displaymath}
        \Prob\left[\|\UA' \Mx{S}' \Mx{S} \Bperp\|_F^2 \geq \frac{\epsilon \resid^2}{2}\right]
        \leq \frac{2\nu}{\tilde{\beta} \epsilon \tilde{s}}.
    \end{displaymath}
    Setting the right-hand side equal to $\delta$ and solving for $\tilde{s}$ yields $\tilde{s} \geq \frac{2\nu}{\tilde{\beta} \delta \epsilon}$.
\end{proof}

\begin{corollary}[Hybrid SC2 in terms of $\beta$]\label{cor:hybrid-sc2}
    Under the same setup as \cref{lem:hybrid-sc2}, if $p_i \geq \beta \ell_i(\A)/r$ for all $i \in [N]$ and some $\beta \in (0,1]$, then
    \begin{displaymath}
        \tilde{s} \geq \frac{2r(1-\pdet)}{\beta \delta \epsilon}
        \quad \Rightarrow \quad
        \Pr\left\{ \| \UA' \Mx{S}' \Mx{S} \Bperp \|_F^2 \leq \frac{\epsilon \resid^2}{2} \right\} \geq 1-\delta.
    \end{displaymath}
\end{corollary}

\begin{proof}
    The assumption $p_i \geq \beta \ell_i(\A)/r$ implies $\tilde{p}_i = p_i/(1-\pdet) \geq \beta \ell_i(\A)/(r(1-\pdet))$.
    Define $\tilde{\beta} = \beta \nu/(r(1-\pdet))$.
    Then $\tilde{p}_i \geq \tilde{\beta} \ell_i(\A)/\nu$, satisfying the hypothesis of \cref{lem:hybrid-sc2}.
    The result follows since
    \begin{displaymath}
        \tilde{s} \geq \frac{2\nu}{\tilde{\beta} \delta \epsilon} = \frac{2\nu  r(1-\pdet)}{\beta \nu \delta \epsilon} = \frac{2r(1-\pdet)}{\beta \delta \epsilon}.
    \end{displaymath}
\end{proof}

\subsection{Main Theorem for Hybrid Sampling}
\label{sec:hybrid-main-theorem}

The following theorem is analogous to \cref{thm:sketching-app} for hybrid leverage score sampling. It uses the hybrid versions of the structural conditions from \cref{lem:hybrid-sc1} and \cref{cor:hybrid-sc2}.

\begin{theorem}[Hybrid sampling]\label{thm:hybrid-sample}
    Consider the least squares problem
    $\min_{\Mx{X} \in \Real^{r \times n}} \| \A \Mx{X} - \B\|^2$
    where $\A \in \Real^{N \times r}$ with $r \ll N$ and $\rank(\A) = r$
    and  $\B \in \Real^{N \times n}$.
    Let $\Vp \in [0,1]^N$ be a probability distribution and
    assume there exists a fixed
    $\beta \in (0,1]$ such that
    \begin{displaymath}
        p_i \geq \beta \frac{\ell_i(\A)}{r} \qtext{for all} i \in [N].
    \end{displaymath}
    Let $\DetSet \subset [N]$ be a deterministic index set with $d = |\DetSet|$ elements.
    For any  $\epsilon, \delta \in (0,1)$,
    set
    \begin{displaymath}
        s = d + \tilde{s}
        \qtext{with}
        \tilde{s} = (1-\pdet) \frac{r}{\beta} \max\set{ 2C \ln(2r/\delta), \frac{4}{\delta\epsilon} },
    \end{displaymath}
    where $C = \sqrt{2}/(\sqrt{2}-1-\log(2)/2) \approx 20.9080$.
    Let $\Mx{S} \sim \HybridSample(s, \Vp, \DetSet)$.
    Define
    $\Xopt \equiv \arg \min_{\Mx{X} \in \Real^{r \times n}} \| \A \Mx{X} - \B\|^2$.
    Then $\Xtildeopt \equiv \arg \min_{\Mx{X} \in \Real^{r \times n}} \| \Mx{S} \A \Mx{X} - \Mx{S} \B\|_F^2$
    satisfies:
    \begin{displaymath}
        \| \A \Xtildeopt - \Mx{B} \|_F^2 \leq (1+\epsilon) \| \A \Xopt - \Mx{B} \|_F^2
    \end{displaymath}
    with probability at least $1-\delta$.
\end{theorem}

\begin{proof}
    By \cref{thm:sketch-bound}, it suffices to verify that the structural conditions \cref{eq:sc1,eq:sc2} both hold with probability at least $1-\delta$.
    We show each holds with probability at least $1-\delta/2$, then apply a union bound.

    For SC1, \cref{cor:hybrid-sc1} guarantees $\sigma_{\min}^2(\Mx{S}\UA) \geq 1/\sqrt{2}$ with probability at least $1-\delta/2$ when
    \begin{displaymath}
        \tilde{s} \geq 2C\frac{r(1-\pdet)}{\beta} \ln(2r/\delta).
    \end{displaymath}

    For SC2, \cref{cor:hybrid-sc2} guarantees $\|\UA'\Mx{S}'\Mx{S}\Bperp\|_F^2 \leq \epsilon \resid^2/2$ with probability at least $1-\delta/2$ when $\tilde{s} \geq 4r(1-\pdet)/(\beta \delta \epsilon)$.

    Taking $\tilde{s}$ to be the maximum of the SC1 and SC2 bounds ensures both conditions hold simultaneously with probability at least $1-\delta$.
\end{proof}

\begin{remark}[Comparison with standard sampling]
    \label{rem:hybrid-comparison}
    Standard leverage score sampling (\cref{thm:sketching-app}) sets
    \begin{displaymath}
        s = \frac{r}{\beta} \max\set{C\ln(2r/\delta), \frac{4}{\delta\epsilon}}
    \end{displaymath}
    versus, in \cref{thm:hybrid-sample},
    \begin{displaymath}
        s = \textcolor{red}{d} + \textcolor{red}{(1-\pdet)} \frac{r}{\beta} \max\set{\textcolor{red}{2}C \ln(2r/\delta), \frac{4}{\delta\epsilon}},
    \end{displaymath}
    with the differences highlighted in red.
    We generally assume $\varepsilon$ is small enough that the second term in the maximum dominates. So, the hybrid sampling reduced the random samples by a factor of $(1-\pdet)$, at the cost of $d$ deterministic rows.
    Thus, hybrid is beneficial if
    \begin{displaymath}
        d + (1-\pdet) \frac{4r}{\beta \delta \epsilon} \ll \frac{4r}{\beta \delta \epsilon}
        \quad \Leftrightarrow \quad
        d \ll \pdet \frac{4r}{\beta \delta \epsilon}.
    \end{displaymath}
\end{remark}

\begin{remark}[Comparison to Results of Hayashi et al.~\cite{hayashi2025randomized}]
    Hayashi et al.~\cite{hayashi2025randomized} analyze hybrid
    sampling for \emph{exact} leverage scores.
    Their key idea was splitting $\UA$ into two parts
    corresponding to the deterministic and random rows.
    We use this key idea to extend
    their results to the case of approximate leverage scores.
    In addition, we provided an updated constant for both results.
\end{remark}

\appendix
\section{Acknowledgments}%
AI tools were used in the development of the results in this manuscript,
especially in providing ideas for the proofs that improved the constants. 
The authors assume responsibility for all content.

\bibliographystyle{siamplainmod}

\begin{thebibliography}{10}

\bibitem{drineas2006fast}
{\sc P.~Drineas, R.~Kannan, and M.~W. Mahoney}, {\em Fast {Monte Carlo}
  algorithms for matrices {I}: Approximating matrix multiplication}, SIAM
  Journal on Computing, 36 (2006), pp.~132--157,
  \href{http://dx.doi.org/10.1137/s0097539704442684}
  {\nolinkurl{doi:10.1137/s0097539704442684}}.

\bibitem{DrMaMaWo12}
{\sc P.~Drineas, M.~Magdon-Ismail, M.~W. Mahoney, and D.~P. Woodruff}, {\em
  Fast approximation of matrix coherence and statistical leverage}, Journal of
  Machine Learning Research, 13 (2012), pp.~3475--3506,
  \url{http://www.jmlr.org/papers/v13/drineas12a.html}.

\bibitem{DrMa17}
{\sc P.~Drineas and M.~W. Mahoney}, {\em Lectures on randomized numerical
  linear algebra}, 2017, \href{http://arxiv.org/abs/1712.08880}
  {arXiv:1712.08880}.

\bibitem{drineas2011faster}
{\sc P.~Drineas, M.~W. Mahoney, S.~Muthukrishnan, and T.~Sarl{\'o}s}, {\em
  Faster least squares approximation}, Numerische mathematik, 117 (2011),
  pp.~219--249, \href{http://dx.doi.org/10.1007/s00211-010-0331-6}
  {\nolinkurl{doi:10.1007/s00211-010-0331-6}}.

\bibitem{hayashi2025randomized}
{\sc K.~Hayashi, S.~G. Aksoy, G.~Ballard, and H.~Park}, {\em Randomized
  algorithms for symmetric nonnegative matrix factorization}, SIAM Journal on
  Matrix Analysis and Applications, 46 (2025), pp.~584--625,
  \href{http://dx.doi.org/10.1137/24m1638355}
  {\nolinkurl{doi:10.1137/24m1638355}}.

\bibitem{HoJo85}
{\sc R.~A. Horn and C.~R. Johnson}, {\em Matrix analysis}, Cambridge University
  Press, 1985.

\bibitem{LaKo22}
{\sc B.~W. Larsen and T.~G. Kolda}, {\em Practical leverage-based sampling for
  low-rank tensor decomposition}, SIAM Journal on Matrix Analysis and
  Applications, 43 (2022), pp.~1488--1517,
  \href{http://dx.doi.org/10.1137/21M1441754}
  {\nolinkurl{doi:10.1137/21M1441754}}.

\bibitem{arXiv-LaKo22}
{\sc B.~W. Larsen and T.~G. Kolda}, {\em Sketching matrix least squares via
  leverage scores estimates}, January 2022,
  \href{http://arxiv.org/abs/2201.10638v1} {arXiv:2201.10638v1 [math.NA]}.

\bibitem{magdon2010row}
{\sc M.~Magdon-Ismail}, {\em Row sampling for matrix algorithms via a
  non-commutative {Bernstein} bound}, 2010,
  \href{http://arxiv.org/abs/1008.0587} {arXiv:1008.0587}.

\bibitem{murray2023randomized}
{\sc R.~Murray, J.~Demmel, M.~W. Mahoney, N.~B. Erichson, M.~Melnichenko, O.~A.
  Malik, L.~Grigori, P.~Luszczek, M.~Derezinski, M.~E. Lopes, T.~Liang, H.~Luo,
  and J.~Dongarra}, {\em Randomized numerical linear algebra: A perspective on
  the field with an eye to software}, Feb 2023,
  \href{http://arxiv.org/abs/2302.11474} {arXiv:2302.11474}.
\newblock Also available as Technical Report No.~UCB/EECS-2023-19, University
  of California, Berkeley.

\bibitem{Tropp15}
{\sc J.~A. Tropp}, {\em An introduction to matrix concentration inequalities},
  Foundations and Trends in Machine Learning, 8 (2015), pp.~1--230,
  \href{http://dx.doi.org/10.1561/2200000048}
  {\nolinkurl{doi:10.1561/2200000048}}.

\bibitem{tropp2020randomized}
{\sc J.~A. Tropp}, {\em Randomized algorithms for matrix computations}, 2020,
  \url{https://tropp.caltech.edu/notes/Tro20-Randomized-Algorithms-LN.pdf}.
\newblock Lecture notes for ACM 204 at Caltech in Winter 2020, prepared by Dr.
  Richard Keung et al.

\bibitem{Wo14}
{\sc D.~P. Woodruff}, {\em Sketching as a tool for numerical linear algebra},
  {FNT} in Theoretical Computer Science, 10 (2014), pp.~1--157,
  \href{http://dx.doi.org/10.1561/0400000060}
  {\nolinkurl{doi:10.1561/0400000060}}.

\end{thebibliography}

\clearpage
\end{document}